\documentclass[11pt,twoside]{article}
\usepackage{latexsym,amsmath,amssymb}

\newif\ifdviwin

\usepackage{titlesec}
	
\titleformat{\section}{\large\bfseries\center}{\thesection}{1em}{\vspace{-.5cm}}
\titleformat{\subsection}[runin]{\bfseries}{\thesubsection}{1em}{}

\usepackage[english]{babel}
\usepackage{indentfirst}
\usepackage[mathscr]{eucal}
\usepackage{amssymb,amsmath,amsfonts}
\usepackage{fancybox,fancyhdr}
\usepackage{graphicx}
\usepackage[utf8]{inputenc}
\usepackage{float}
\usepackage{color}
\usepackage[pdftex]{hyperref}
\hypersetup{colorlinks=true,linkcolor=blue,citecolor=blue} 

\newif\ifdviwin

\dviwintrue

\def\H{\mathcal{H}}
\def\H{\mathfrak{h}}

\def\m2r{\mathbb{M}^2\times\mathbb{R}}
\def\h2r{\mathbb{H}^2\times\mathbb{R}}

\let\infty=\infty \let\0=\emptyset  
\let\hat=\widehat

\let\tilde=\widetilde

\let\varepsilon=\varepsilon

\def\cte.{\mathop{\rm cte.}\nolimits}

\def\E{\mathbb{E}}

\def\R{\mathbb{R}}

\def\M{\mathbb{M}}

\def\H{\mathcal{H}}

\def\D{\mathbb{D}}

\def\S{\mathbb{S}}

\def\m2r{\M^2\times\R}
\def\mkr{\M^2(\kappa)\times\R}
\def\Ekt{\E(\kappa,\tau)}
\def\rkt{\mathcal{R}(\kappa,\tau)}
\def\sb{\mathbb{S}^3_{b}(\kappa,\tau)}
\def\h2r{\mathbb{H}^2\times\R}
\def\s2r{\mathbb{S}^2\times\R}

\def\nil{\mathrm{Nil}_3}
\def\sl{\tilde{SL_2}(\R)}
\def\Hs{\mathfrak{h}\text{-}\mathrm{surface}}
\def\Hss{\mathfrak{h}\text{-}\mathrm{surfaces}}

\def\sig{\Sigma}
\def\r3{\mathbb{R}^3}
\def\arctanh{\mathrm{arctanh}}
\def\h{\mathfrak{h}}

\def\c{\mathfrak{C}^1}
\def\radi{4\h(y)^2+\kappa(1-y^2)+4\tau^2y^2}
\def\raiz{\sqrt{4\h(y)^2+\kappa(1-y^2)+4\tau^2y^2}}
\def\t1{\Theta_1}
\def\tm1{\Theta_{-1}}
\def\te{\Theta_\varepsilon}

 \newtheorem{defi}{Definition}[section]
 \newtheorem{teo}[defi]{Theorem}
 \newtheorem{pro}[defi]{Proposition}

 \newtheorem{obs}[defi]{Observation}

 \newenvironment{proof}{\rm \trivlist \item[\hskip \labelsep{\it
      Proof}:]}{\nopagebreak \hfill $\Box$ \endtrivlist}

 \newenvironment{proofclaim}{\rm \trivlist \item[\hskip \labelsep{\it
      Proof of the claim }:]}{\nopagebreak \hfill $\Box$ \endtrivlist}

\numberwithin{equation}{section}

\begin{document}
\thispagestyle{empty}

\begin{center}

\renewcommand{\thefootnote}{\,}
{\Large \bf Delaunay surfaces of prescribed mean curvature in $\nil$ and $\sl$
\footnote{\hspace{-.75cm}
}}\\
\vspace{0.5cm} { Antonio Bueno}\\
\end{center}
\vspace{.5cm}
Departamento de Ciencias, Centro Universitario de la Defensa de San Javier, E-30729 Santiago de la Ribera, Spain. \\ 
\emph{E-mail address:} antonio.bueno@cud.upct.es

\begin{abstract}
We obtain a classification result for rotational surfaces in the Heisenberg space and the universal cover of the special linear group, whose mean curvature is given as a prescribed $C^1$ function depending on their angle function. We show that these surfaces behave like the Delaunay surfaces of constant mean curvature, under some assumptions on the prescribed function. In contrast with the constant mean curvature case, we exhibit the existence of rotational, embedded tori, providing counterexamples of the Alexandrov problem for this class of immersed surfaces.
\end{abstract}

\section{Introduction}
The study of surfaces of positive constant mean curvature (CMC surfaces) in the Euclidean space $\r3$ is a widely studied topic in the past centuries, whose impact transcends other fields in mathematics such as Analysis of PDE's, Complex Analysis, Geometric Measure Theory and Topology. Among the classical results we highlight the \emph{Delaunay theorem}, that classifies the complete, rotational CMC surfaces as round spheres, cylinders, unduloids and nodoids. In the literature, these surfaces are known as \emph{Delaunay surfaces}.

The theory of CMC surfaces has been extended to further ambient spaces, being of remarkable importance the $\Ekt$ spaces: the homogeneous, simply connected 3-dimensional manifolds whose isometry group have dimension greater than three and are not space forms. One of the major achievements was the resolution of the Hopf problem by Abresch-Rosenberg \cite{AbRo1,AbRo2}, proving that the only immersed CMC spheres are the rotational ones. This milestone attracted the attention of many researches, becoming an active and fruitful field of research; see e.g. \cite{Dan, DHM, FeMi} and references therein for an outline of the development of this theory.

In the past decades, many celebrated results of the theory of CMC surfaces have been generalized to the $\Ekt$ spaces, among which we remark the classification of complete, rotational CMC surfaces. It was achieved in the product spaces $\mkr$ by \cite{HsHs,PeRi}; in the Heisenberg space $\nil$ by \cite{Tom}; in the universal cover of the special linear group $\sl$ by \cite{Gor,Tor}; and in the Berger spheres $\S^3_b$ by \cite{Tor}. If the constant mean curvature $H_0$ satisfies $4H_0^2+\kappa>0$, the rotational CMC surfaces behave as the Delaunay surfaces in $\r3$. The value $\sqrt{-\kappa}/2$ for $\kappa\leq 0$ is known as the \emph{critical value} of the mean curvature; there exists a rotational sphere with CMC equal to $H_0$ if and only if $H_0>\sqrt{-\kappa}/2$.  

Taking as main motivation the Delaunay classification of rotational CMC surfaces in the $\Ekt$ spaces, in this paper we extend this result to the following class of immersed surfaces:
\begin{defi}\label{defihsup}
Let be $\h\in C^1([-1,1])$. An immersed, oriented surface $\sig$ in $\Ekt$ has \emph{prescribed mean curvature} $\h$ if its mean curvature function $H_\sig$ is given by
\begin{equation}\label{eqhsup}
H_\sig(p)=\h(\langle\eta_p,\xi\rangle),\hspace{.5cm} \forall p\in\sig,
\end{equation}
where $\eta$ is the unit normal of $\sig$, $\xi$ is the vertical Killing vector field in $\Ekt$ and $\langle\eta_p,\xi\rangle=:\nu_p$ is the angle function.
\end{defi}
For short, we will say that $\sig$ is an $\h$-surface. 

Some $\Hss$ have already appeared in the literature for certain prescribed functions: if $\h=H_0\in\R$ we have CMC surfaces and if $\h(y)=y$ we have translating solitons \cite{Bue1,Bue2,LiMa,Pip}. Also, for a general $\h$ the author has started the study of global properties of $\Hss$ in $\mkr$ \cite{Bue3,Bue4}.


The definition of this class of surfaces has its roots in the following \emph{Minkowsky-type prescribed curvature problem} in $\r3$:
\begin{defi}
Let be $\H\in C^1(\S^2)$. An immersed, oriented surface $\sig$ in $\r3$ is an $\H$-surface if its mean curvature function $H_\sig$ is given by
\begin{equation}\label{PMCEQR3}
H_\sig(p)=\H(N_p),\hspace{.5cm} \forall p\in\sig,
\end{equation}
where $N:\sig\rightarrow\S^n$ is the \emph{Gauss map} of $\sig$.
\end{defi}
Again, for $\H=H_0$ a constant we recover the surfaces with constant mean curvature $H_0$.

The study of surfaces defined by a prescribed relation between its principal curvatures and its Gauss map goes back, at least, to the famous Minkowski and Christoffel problems for ovaloids \cite{Min,Chr}. When we prescribe the mean curvature as in Equation \eqref{PMCEQR3}, Alexandrov and Pogorelov in the '50s \cite{Ale,Pog} and more recently Guan-Guan \cite{GuGu} and Gálvez-Mira \cite{GaMi1,GaMi2,GaMi3}, among others, focused on the existence and uniqueness of immersed $\H$-spheres. Recently, the author jointly with Gálvez and Mira started to develop the \emph{global theory of hypersurfaces with prescribed mean curvature} in \cite{BGM1,BGM2}, taking as starting point the well-studied theory of CMC surfaces in $\r3$. In \cite{BGM1} the authors focused on the study of rotational $\H$-surfaces, and in \cite{BGM2} global structure results for properly embedded surfaces, including height and curvature estimates, were exhibited.

We highlight the \emph{Delaunay-type classification result} achieved in \cite{BGM1}. Under necessary and sufficient assumptions on the prescribed function, the authors proved that the rotational $\H$-surfaces behave the same as the CMC Delaunay surfaces in $\r3$. The arbitrariness of the prescribed function made hopeless to find a first integral, even for concrete choices. Instead, the geometric properties of the rotational $\H$-surfaces were analyzed by means of a \emph{phase plane} analysis.  Inspired by these techniques, the author achieved a Delaunay-type classification result for $\h$-surfaces in $\mkr$ \cite{Bue3}.

As happened for $\H$-surfaces in $\r3$, some hypotheses on $\h$ are needed if we expect a Delaunay behavior for rotational $\Hss$. Throughout this paper, the prescribed function $\h$ will be always assumed to belong to the following set of functions:
\begin{equation}\label{espaciofunciones}
\mathfrak{C}^1:=\{\h\in C^1([-1,1]);\ \h(y)=\h(-y)>0\ \text{and}\ 4\h(y)^2+\kappa(1-y^2)>0,\ \forall y\in [-1,1]\}.
\end{equation}
Note that for the particular case that $\h$ is a constant $H_0>0$ and $\kappa<0$, the fact that $H_0\in\c$ reads as $H_0>\sqrt{-\kappa}/2$, which agrees with the critical value of the mean curvature.

Our main result in this paper is the following Delaunay-type classification result for $\Hss$.
\begin{teo}\label{teoremadelaunay}
Let be $\h\in\c$. Up to vertical translations, any complete, rotational $\h$-surface in $\nil$ or $\sl$ is one of the following:
\begin{itemize}
\item[1.] The vertical cylinder with radius $x_0=2\left(\sqrt{4\h(0)^2+\kappa}+2\h(0)\right)^{-1}$ and constant mean curvature $\h(0)$.
\item[2.] An embedded $\h$-sphere with strictly monotone angle function.
\item[3.] A 1-parameter family of properly embedded $\h$-unduloids.
\item[4.] A 1-parameter family of properly immersed $\h$-nodoids.
\end{itemize}
\end{teo}

One of the major issues is that the $\h$-nodoids of type $\mathit{4.}$ in Theorem \ref{teoremadelaunay} may end up closing. This configuration would lead to the existence of rotational embedded $\h$-tori, providing counterexamples to the \emph{Alexandrov problem} for $\Hss$. In general, the Alexandrov problem for a class of surfaces asks whether a compact, embedded surface is topologically a sphere. This problem was originally posed by Alexandrov for CMC surfaces in $\r3$ and proved by applying his celebrated argument of moving planes. 

For CMC surfaces in $\Ekt$, the Alexandrov problem is fully solved in $\mkr$ and $\S^3_b$. In the Heisenberg space $\nil$ and the space $\sl$, the Alexandrov problem is still an outstanding, major open problem. Regarding $\Hss$, the Alexandrov problem in $\mkr$ and $\S^3_b$ is also solved. The main difference here is that for $\kappa\leq0$ and $\tau\neq0$, we give sufficient conditions for the non-existence of rotational $\h$-tori and also ensure the existence of rotational $\h$-tori for certain prescribed functions $\h\in\c$.

%
\begin{teo}\label{teorematoros}
Let be $\h\in\c$	.
\begin{itemize}
\item[1.] If $\h$ is non-increasing in $[-1,0]$, then there do not exist rotational $\h$-tori in $\nil$ and $\sl$.
\item[2.] There are choices of $\h$ such that there exist rotational embedded $\h$-tori in $\nil$ and $\sl$
\end{itemize}
\end{teo}
Note that a constant $\h=H_0\in\c$ lies in the hypothesis of Item $\mathit{1.}$, recovering the non-existence of rotational CMC tori in $\nil$ and $\sl$. 
 
These rotational and embedded $\h$-tori provide counterexamples to Alexandrov problem for $\Hss$ in $\nil$ and $\sl$, that is the rotational embedded $\h$-spheres given by Item $\mathit{1.}$ of Theorem \ref{teoremadelaunay} are not unique in Alexandrov sense.

The rest of the introduction is devoted to further detail the organization of the paper.

In Section \ref{secpropbasicas} we define the $\Ekt$ spaces as the family of 3-dimensional homogeneous manifolds with a 4-dimensional isometry group, and we introduce a canonical coordinate model. We also define the class of immersed $\Hss$ in $\Ekt$ and deduce some properties, making special emphasis on the ambient isometries that are also isometries for $\Hss$.

In Section \ref{rotational} we focus in the analysis of rotational $\Hss$. In the same fashion as in \cite{BGM1}, the approach will be done by means of a phase plane study of the solutions of the non-linear, autonomous system of ODE's that the coordinates of the profile curve satisfy. In Section \ref{basicformulas} we deduce the formulas that the profile curve satisfies, and relate the geometry of this curve with the prescribed mean curvature. In Section \ref{thephaseplane} we define the phase plane and exhibit some of its first elements. Sections \ref{behaviorofgamma} and \ref{structurephaseplane} are devoted to study in detail the local and global properties of this phase plane.


Finally, Section \ref{clasificacionrotacionales} is fully devoted to the proof of Theorem \ref{teoremadelaunay}, while in Section \ref{discusiontoros} we prove Theorem \ref{teorematoros}. Specifically, in Section \ref{nonexistencetori} we give sufficient conditions on the prescribed function $\h$ for the non-existence of rotational $\h$-tori, and in Section \ref{existencetori} we exhibit the existence of $\h$-tori.
 
\textbf{Acknowledgments:} The author is thankful to José A. Gálvez, José M. Manzano and Francisco Torralbo for helpful comments and observations.

\section{Immersed $\h$-surfaces in the $\Ekt$ spaces}\label{secpropbasicas}

\subsection{The $\Ekt$ spaces.}\label{TheEKTspaces} Consider a homogeneous, simply connected, 3-dimensional manifold whose isometry group has dimension greater than 3 and that is not a space form. Then, its isometry group has dimension 4 and is one of the $\Ekt$ spaces for some $\kappa,\tau\in\R$ such that $\kappa\neq4\tau^2$, see \cite{Dan}. A change in the orientation in the space changes $\tau$ into $-\tau$, hence we will suppose that $\tau>0$ without losing generality.

The $\Ekt$ spaces admit a Riemannian submersion $\pi:\Ekt\rightarrow\M^2(\kappa)$ onto the complete, simply connected surface of constant curvature $\kappa$. This fibration has a unitary Killing vector field that will be denoted by $\xi$, whose integral curves are precisely the fibers of the submersion; recall that the fibers are the sets $\pi^{-1}(q),\ q\in\M^2(\kappa)$. The group of isometries generated by the Killing vector field $\xi$ are the \emph{vertical translations}. 

If $\tau=0$ we recover the product spaces $\M^2(\kappa)\times\R$ and the submersion $\pi$ is isomorphic to the projection $\M^2(\kappa)\times\R\rightarrow\M^2(\kappa)$. When $\tau>0$ we get the Heisenberg space $\nil$ for $\kappa=0$; the Berger spheres $\sb$ for $\kappa>0$; and the universal cover of the special linear group, the space $\sl$, for $\kappa<0$. Note that in the Berger spheres, the projection $\pi$ is isomorphic to the Hopf fibration. 

A key feature is that for every $p\in\Ekt$ there exists a continuous 1-parameter family of orientation preserving isometries leaving pointwise fixed the fiber $\pi^{-1}(\pi(p))$; these isometries will be called \emph{rotations around the axis $\pi^{-1}(\pi(p))$}.

Next we describe a coordinate model for the $\Ekt$ spaces; when $\kappa\leq 0$ the model is global, and when $\kappa>0$ the model is homeomorphic to the universal cover of the space minus one fiber. The notation used is inspired by Section 2.1. in \cite{GaMi3}. We consider $\rkt$ to be the space $\r3$ if $\kappa\geq 0$, or the disk $\D(2/\sqrt{-\kappa})$ if $\kappa<0$, endowed with coordinates $(x,y,z)$, and the metric
\begin{equation}\label{metricaE}
\langle\cdot,\cdot\rangle=\lambda^2(dx^2+dy^2)+\left(\lambda\tau(ydx-xdy)+dz\right)^2,\hspace{.5cm} \lambda=\frac{4}{4+\kappa(x^2+y^2)}.
\end{equation}
Then, $\rkt$ is isometric to the corresponding $\Ekt$ space, and the Riemannian submersion is isomorphic to the projection onto the first two coordinates. The vector fields
$$
E_1=\frac{1}{\lambda}\partial_x-\tau y\partial_z,\hspace{.5cm} E_2=\frac{1}{\lambda}\partial_y+\tau x\partial_z,\hspace{.5cm} E_3=\xi=\partial_z,
$$
are an orthonormal frame. The $E_3$-axis is defined to be the fiber $\pi^{-1}(\pi((0,0,0)))$. In this coordinate model we have that the usual rotations 
$$
(x,y,z)\longmapsto (x\cos\theta+y\sin\theta,-x\sin\theta+y\cos\theta,z),\hspace{.5cm} \theta\in\R,
$$
are the rotations around the $E_3$-axis.

Given an immersed surface $\sig$ in $\Ekt$ and $\eta:\sig\rightarrow T\Ekt$ a unit normal along $\sig$, the function defined by
\begin{equation}\label{anguloekt}
\nu:\sig\rightarrow\R,\hspace{.5cm} \nu_p=\langle\eta_p,E_3\rangle
\end{equation}
is the \emph{angle function} of $\sig$. 

\subsection{Isometries and $\h$-surfaces.}\label{hsurfacesekt} As stated in Definition \ref{defihsup} in the Introduction, given $\h\in C^1([-1,1])$ an $\Hs$ in an $\Ekt$ space is an immersed surface $\sig$ whose mean curvature satisfies $H_\sig(p)=\h(\nu_p)$, for every $p\in\sig$.

Next we describe the isometries in the $\Ekt$ spaces that send $\Hss$ into $\Hss$. As a matter of fact, that translations and rotations around the fibers leave invariant the angle function at any $\Ekt$ space, hence in virtue of Equation \eqref{eqhsup} send $\Hss$ into $\Hss$.

The only isometries in the $\Ekt$ spaces that change the value of $\nu$ are reflections with respect to horizontal planes if $\tau=0$ and rotations of angle $\pi$ around horizontal geodesics if $\tau\neq0$; any of these isometries change the value $\nu$ into $-\nu$. Hence, if $\h$ is an even function, $\sig$ is an $\Hs$ and $\Psi$ is any of those isometries, then $\Psi(\sig)$ also satisfies Equation \eqref{defihsup} and so $\Psi$ send $\Hss$ into $\Hss$.

\textbf{2.3 Existence of radial solutions.} In this section we show the existence of radial solutions for graphical $\Hss$ defined over a disk with small enough radius. The following proposition is a consequence of a more general existence result for radial solutions of a fully non-linear PDE in the $\Ekt$ spaces, see Lemma 4.1 in \cite{GaMi3}.
\begin{pro}\label{problemadirichlet}
Let be $\h\in C^1([-1,1])$. There exists $\delta>0$ and a function $f:[0,\delta]\rightarrow\R$ such that the radial graph over the distance disk $D(\textbf{o},\delta)$ in $\M^2(\kappa)$,
$$
\sig_f:=\{(x\cos\theta,x\sin\theta,f(x));\ x\in[0,\delta],\ \theta\in [0,2\pi]\},\hspace{.5cm} f'(0)=0,
$$
with upwards orientation is an $\h$-surface in $\rkt$. Moreover, $\sig_f$ is unique among graphical $\h$-surfaces over $D(\textbf{o},\delta)$ having constant boundary data. The same holds for downwards orientation.
\end{pro}

\section{Rotational $\h$-surfaces in homogeneous 3-spaces}\label{rotational}

\subsection{Basic formulas.}\label{basicformulas} We begin by locally parametrizing a rotational $\Hs$ in the coordinate model $\rkt$. Let $\alpha(u)=(x(u),0,z(u))$ be a curve in the $xz$-plane\footnote{The $xz$-plane is just the subset $\{y=0\}$ in $\rkt$. Only when $\tau=0$ it is a totally geodesic surface in $\Ekt$ isometric to $\R^2$.}. The map
$$
\psi(u,\theta)=\left(x(u)\cos\theta,x(u)\sin\theta,z(u)\right),\hspace{.5cm}
$$
defines an immersed surface $\sig$ as the image of $\alpha(u)$ under the rotations of $\rkt$ that leave the $E_3$-axis pointwise fixed. Note that the projection $\pi(\psi(u,\theta))$ lies in the section $\{z=0\}\subset\rkt$. For $\kappa\geq 0$ this reads as $x(u)\in(0,\infty)$ (since the section $\{z=0\}$ is $\R^2$), but for $\kappa<0$ this implies that $x(u)\in(0,2/\sqrt{-\kappa})$.

The angle function of $\sig$ in this model is
$$
\nu=\frac{4x'}{\sqrt{16(1+\tau^2x^2)x'^2+z'^2(4+\kappa x^2)^2}}.
$$
and the mean curvature $H_\sig$ has the following expression

\begin{equation}\label{eqmedia}
2H_\sig=\frac{\left(4+\kappa x^2\right)^2 \left(z'^3 \left(16-\kappa^2 x^4\right)-16 z' \left(\tau^2 x^3 x''+x x''-x'^2\right)+16 z''x x'\left(1+\tau^2 x^2\right)\right)}{4 x \left(z'^2 \left(4+\kappa x^2\right)^2+16 x'^2 \left(1+\tau^2 x^2\right)\right)^{3/2}}.
\end{equation}
Now we consider the metric 
\begin{equation}\label{metricappa}
d\sigma^2=(1+\tau^2x^2)dx^2+\frac{(4+\kappa x^2)^2}{16} dz^2
\end{equation}
in the $xz$-plane and the arc-length parameter $s$ of $\alpha$ with respect to this metric. A straightforward computation shows that, with this arc-length parameter, the angle function is $\nu=x'$ and the mean curvature is
$$
2\varepsilon H_\sig=\frac{x \left(-x'' \left(4+\kappa x^2\right) \left(1+\tau^2 x^2\right)+x x'^2 \left(\kappa-8 \tau^2\right)-\kappa x\right)-4 x'^2+4}{4 x \sqrt{1-x'^2 \left(1+\tau^2 x^2\right)}},\hspace{.5cm} \varepsilon:=\mathrm{sign}(z').
$$
From now on we suppose that $\sig$ is an $\h$-surface for some $\h\in C^1([-1,1])$, that is $H_\sig(p)=\h(\nu_p),\ \forall p\in\sig$. Solving this equation for $x''$ yields
$$
x''=\frac{4-\kappa x^2-x'^2(4-x^2(\kappa-8\tau^2))-8\varepsilon x\h(x')\sqrt{1-(1+\tau^2 x^2)x'^2}}{x \left(4+\kappa x^2\right) \left(1+\tau^2 x^2\right)}.
$$

After the change of variable $x'=y$, this equation transforms into the first order, autonomous system
\begin{equation}\label{1ordersys}
\left(\begin{array}{c}
x\\
y
\end{array}\right)'=\left(\begin{array}{c}
y\\
\displaystyle{\frac{4-\kappa x^2-y^2(4-x^2(\kappa-8\tau^2))-8\varepsilon x\h(y)\sqrt{1-(1+\tau^2 x^2)y^2}}{x \left(4+\kappa x^2\right) \left(1+\tau^2 x^2\right)}}
\end{array}\right).
\end{equation}

From the arc-length condition $(1+\tau^2x^2)x'^2+(4+\kappa x^2)^2/16z'^2=1$ we have that the angle function $y=x'$ satisfies
$$
\frac{-1}{\sqrt{1+\tau^2x^2}}\leq y\leq\frac{1}{\sqrt{1+\tau^2x^2}},
$$
with equality if and only if the height function $z$ of $\alpha$ has a local extremum. This implies that system \eqref{1ordersys} is only defined for points $(x_0,y_0)$ such that $x_0>0$ and $y_0^2\leq 1/(1+\tau^2x_0^2)$. For instance, note that for the case $\tau> 0$, the angle function satisfies $y=\pm1$ if and only if $x=0$, which only happens at the axis of rotation.

\subsection{The phase plane.}\label{thephaseplane} Hereinafter we suppose that $\tau>0$ and $\kappa\leq0$, i.e. we focus on the spaces $\nil$ and $\sl$.

The phase plane of Equation \eqref{1ordersys} is defined as the set
$$
\Theta_\varepsilon:=\left\lbrace(x,y);\ x>0\ \mathrm{and}\ y^2<\frac{1}{1+\tau^2x^2}\right\rbrace,
$$
with coordinates $(x,y)$ denoting the distance to the axis of rotation and the angle function. Note that in virtue of the models introduced in Section \ref{TheEKTspaces}, if $\kappa<0$ the $x$-coordinate is defined for $0<x<2/\sqrt{-\kappa}$, while if $\kappa=0$ the $x$-coordinate is defined for every $x>0$. 

The boundary of $\Theta_\varepsilon$ consists of the segment $\{0\}\times[-1,1]$ and the vertical graphs $y=\pm1/\sqrt{1+\tau^2x^2}$. We will denote by $\Omega^+$ (resp. $\Omega^-$) to the component $y=1/\sqrt{1+\tau^2x^2}$ (resp. to the component $y=-1/\sqrt{1+\tau^2x^2}$), and by $\Omega:=\Omega^+\cup\Omega^-$.


The orbits are the solutions of system \eqref{1ordersys} and will be denoted by $\gamma(s)=(x(s),y(s))$. The existence and uniqueness of the Cauchy problem associated to Equation \eqref{1ordersys} has as consequence two important facts: $i)$ two different orbits cannot intersect in $\Theta_\varepsilon$, and $ii)$ the orbits are a foliation of $\Theta_\varepsilon$ by regular $C^1$ curves.
 
Although we will remind it in the statement of the main results, \textbf{$\h$ will be always supposed to lie in the space $\c$}, see Equation \eqref{espaciofunciones} for a definition of the space $\c$.  

For example, the fact that $\h$ is even has the following consequence on the phase plane $\te$: if $\gamma(s)=(x(s),y(s))$ is a solution to Equation \eqref{1ordersys}, so it is $\tilde{\gamma}(s)=(x(-s),-y(-s))$. This condition is related with the fact that rotations of angle $\pi$ around horizontal geodesics if $\tau>0$, are isometries for $\Hss$.

Since $\h\in\c$, a trivial solution to system \eqref{1ordersys} in $\Theta_1$ is the one given by the constant orbit 
\begin{equation}\label{ecuequilibrio}
x(s)=\frac{2}{\sqrt{4\h(0)^2+\kappa}+2\h(0)},\hspace{.5cm} y(s)=0.
\end{equation}
This point is the \emph{equilibrium} of \eqref{1ordersys} and will be denoted by $e_0$. This equilibrium generates a vertical, circular cylinder with constant mean curvature equal to $\h(0)$. In the next section we will see that no equilibria exist in $\tm1$.

From Equation \eqref{1ordersys} we see that the points in $\Theta_\varepsilon$ with $y'(s)=0$ are the ones lying in the intersection of $\Theta_\varepsilon$ with the (possibly disconnected) horizontal graph:
\begin{equation}\label{ecugamma}
x=\Gamma_\varepsilon(y):=2\sqrt{\frac{1-y^2}{\kappa(1-y^2)+8(\h(y)^2+\tau^2 y^2)+4\varepsilon\h(y)\sqrt{4\h(y)^2+\kappa(1-y^2)+4\tau^2 y^2}}}.
\end{equation}
We define $\Gamma_\varepsilon:=\{x=\Gamma_\varepsilon(y)\}\cap\Theta_\varepsilon$. Note that the points lying in $\Gamma_\varepsilon$ correspond to points whose angle function has vanishing derivative, and that $\Gamma_1\cap\{y=0\}=e_0$ and $\Gamma_{-1}\cap\{y=0\}=e_{-1}$ for $\kappa>0$. Again, since $\h$ is even we get that $\Gamma_\varepsilon$ is symmetric with respect to the axis $y=0$.

\begin{obs}
We must clarify the difference between existing an equilibrium point in $\tm1$ and the point given by $(\Gamma_{-1}(0),0)$. Even though the latter can exist (for the case that $\Gamma_{-1}(0)$ is well defined), the point $(\Gamma_{-1}(0),0)$ may not lie in $\tm1$. For instance, if $\kappa<0$ we will see that $\Gamma_{-1}(0)>2/\sqrt{-\kappa}$, hence it lies \emph{outside} $\tm1$. 
\end{obs}

\subsection{The structure of $\Gamma_\varepsilon$.}\label{behaviorofgamma} As revealed in the study made in \cite{BGM1}, the curve $\Gamma_\varepsilon$ deeply governs the behavior of the orbits in $\Theta_\varepsilon$. We study next the properties of the curve $\Gamma_\varepsilon$ in the different phase planes of the spaces $\nil$ and $\sl$.

First, recall that the range of the $x$-coordinate in $\te$ depends on the value of $\kappa$. For the case that $\kappa=0$, the phase plane $\Theta_\varepsilon$ is defined for every $x>0$. For $\kappa<0$, $\Theta_\varepsilon$ is defined for $0<x<2/\sqrt{-\kappa}$. Because of the restrictions on $x$ for $\kappa<0$, the curve $\Gamma_\varepsilon$ can leave $\te$, or even not exist. Nonetheless, we will see that $\Gamma_1$ is always contained in $\t1$.

\textbf{Claim 1.} The curve $\Gamma_1$ for $\kappa<0$ lies entirely in $\Theta_1$, i.e. $\Gamma_1(y)<2/\sqrt{-\kappa},\ \forall y\in[-1,1]$.

\begin{proofclaim}
From \eqref{ecugamma} we see that $\Gamma_1(y)<2/\sqrt{-\kappa}$ if and only if
$$
-\kappa(1-y^2)<\kappa(1-y^2)+8(\h(y)^2+\tau^2y^2)+4\h(y)\sqrt{\radi}.
$$
Simplifying we arrive to
$$
0<4\h(y)^2+\kappa(1-y^2)+4\tau^2y^2+2\h(y)\raiz.
$$
Since $\h\in\c$ in particular $4\h(y)^2+\kappa(1-y^2)>0$ holds for every $y\in[-1,1]$, hence the above inequality yields.
\end{proofclaim}
In any case, the curve $\Gamma_1$ is a connected, compact arc in $\Theta_1$ and having the points $(0,\pm1)$ as endpoints. 

Now, we focus on the curve $\Gamma_{-1}$ in the different spaces. The proof of the following claim holds immediately by just substituting in Equation \eqref{ecugamma} the value $\kappa=0$.

\textbf{Claim 2.} The curve $\Gamma_{-1}$ for $\kappa=0$ is a disconnected bi-graph over the axis $y=0$, having the points $(0,\pm1)$ as endpoints and an asymptote at $y=0$.

As we pointed out, for $\kappa<0$ the curve $\Gamma_{-1}$ may leave the phase plane $\tm1$. The following claim reveals that the point $(\Gamma_{-1}(0),0)$ is well defined, and lies outside $\tm1$.

\textbf{Claim 3.} The value $\Gamma_{-1}(0)$ for $\kappa=-1$ is well defined and $\Gamma_{-1}(0)>2/\sqrt{-\kappa}$.

\begin{proofclaim}
Substituting $\Gamma_{-1}(0)$ in Equation \eqref{ecugamma} and simplifying, we get
$$
\Gamma_{-1}(0)=\displaystyle{\frac{2}{2\h(0)-\sqrt{4\h(0)^2+\kappa}}}.
$$
Since $\kappa<0$, the denominator in the above fraction is always well-defined and so it is the value $\Gamma_{-1}(0)$. For proving that $\Gamma_{-1}(0)>2/\sqrt{-\kappa}$, after a similar computation as in Claim 1. we get
$$
0>4\h(0)^2+\kappa-2\h(0)\sqrt{4\h(0)^2+\kappa}=\sqrt{4\h(0)^2+\kappa}\left(\sqrt{4\h(0)^2+\kappa}-2\h(0)\right).
$$
This time, since $\kappa<0$ we have that the latter expression is negative, concluding the proof.
\end{proofclaim}

By the previous claims, in the phase plane $\tm1$ of both $\nil$ and $\sl$, the curve $\Gamma_{-1}$ is disconnected, has as endpoints $(0,\pm1)$ and does not intersect the axis $y=0$. In $\nil$, $\Gamma_{-1}$ converges to $y=0$ and in $\sl$, $\Gamma_{-1}$ \emph{leaves} $\tm1$ before intersecting the axis $y=0$. In particular, no equilibria exist in $\tm1$.

In both spaces, the curve $\Gamma_\varepsilon$ and the axis $y=0$ divide $\Theta_\varepsilon$ into four connected components, which we will call \emph{monotonicity regions}, where the coordinates $(x(s),y(s))$ of any orbit are monotonous functions; see Figure \ref{planosfases}. Hence the behavior of an orbit is uniquely determined by the monotonicity region where it belongs. We describe this behavior next.

\begin{figure}[H]
\centering
\includegraphics[width=.95\textwidth]{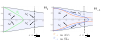} 
\caption{The phase planes $\te,\ \varepsilon=\pm1$ and their monotonicity regions. Note that in $\sl$, $\te$ is only defined for $x<2/\sqrt{-\kappa}$. The arrows indicate the motion of an orbit.}
\label{planosfases}
\end{figure}

\begin{pro}\label{direccionesmonotonia}
Let be $(x_0,y_0)\in \Theta_\varepsilon$ and consider an orbit $\gamma(s)=(x(s),y(s))$ such that $\gamma(s_0)=(x_0,y_0)$. Then, the following properties hold:
\begin{enumerate}
\item If $y_0=0$, then $\gamma$ is orthogonal to the axis $y=0$. If $y_0\neq 0$, then we can see $\gamma(s)$ locally around $\gamma(s_0)$ as a graph $y(x)$. Then:
\item If $x_0>\Gamma_\varepsilon(y_0)$ (resp. $x_0<\Gamma_\varepsilon(y_0)$) and $y_0>0$, then $y(x)$ is strictly decreasing (resp. increasing) at $x_0$. 
\item If $x_0>\Gamma_\varepsilon(y_0)$ (resp. $x_0<\Gamma_\varepsilon(y_0)$) and $y_0<0$, then $y(x)$ is strictly increasing (resp. decreasing) at $x_0$.
\item If $x_0=\Gamma_\varepsilon(y_0)$, then $y'(x_0)=0$ and $y(x)$ has a local extremum at $x_0$.
\end{enumerate}
\end{pro}


\begin{proof}
First, we study how an orbit intersects the axis $y=0$. Suppose that $\varepsilon=1$ and let $\gamma(s)=(x(s),y(s))$ be an orbit in $\Theta_1$ such that  $\gamma(0)=(x_0,0)$ for $x_0>0$. Moreover, suppose that $(x_0,0)$ is not the equilibrium point. From Equation \eqref{1ordersys} we get
$$
\gamma'(0)=\left(0,\frac{4-x_0(8\h(0)+\kappa x_0)}{x_0(4+\kappa x_0^2)(1+\tau^2x_0^2)}\right).
$$
So, $\gamma'(0)$ intersects orthogonally $y=0$, and it does either upwards or downwards depending on the sign of $p_\kappa(x_0)$, where $p_\kappa(x)=4-x(8\h(0)+\kappa x)$. Note that if $\kappa< 0$, the zeroes of $p_\kappa(x)$ are
$$
x_\pm=\frac{2}{2\h(0)\pm\sqrt{4\h(0)^2+\kappa}},
$$
and the point $(x_+,0)\in\t1$ agrees with the equilibrium $e_0$ defined in \eqref{ecuequilibrio}. Depending on the values of $\kappa$, this parabola behaves as follows:

\begin{itemize}
\item For $\kappa=0$, $p_\kappa(x)$ reduces to the straight line $y(x)=4-8\h(0)x$, which has the point $x_+=1/(2\h(0))$ as zero.
\item For $\kappa<0$, $x_-,x_+>0$, $x_+<x_-$, and $p_\kappa(x)>0$ for $x\in(0,x_+)$. Moreover, $x_+<2/\sqrt{-\kappa}<x_-$.
\end{itemize}


Now, we go back to the study of $\gamma'(0)$. First, suppose that $(x_0,0)$ lies at the left hand side of $e_0$, i.e. $x_0<x_+$. Hence, $p_\kappa(x_0)>0$ and thus $\gamma'(0)$ intersects orthogonally the axis $y=0$ pointing upwards. Analogously, $\gamma'(0)$ points downward whenever $(x_0,0)$ lies at the right-hand side of $e_0$, i.e. when $x_0>x_+$.

In $\Theta_{-1}$ the situation is similar. Again, let $\gamma(s)$ be an orbit in $\Theta_{-1}$ such that $\gamma(0)=(x_0,0)$ for $x_0>0$. This time, the value $\gamma'(0)$ is given by
$$
\gamma'(0)=\left(0,\frac{4+x_0(8\h(0)-\kappa x_0)}{x_0(4+\kappa x_0^2)(1+\tau^2x_0^2)}\right),
$$
hence its behavior is determined by the parabola $q_\kappa(x)=4+x(8\h(0)-\kappa x)$. Now, if $\kappa< 0$, $q_\kappa(x)$ has as zeros $z_\pm=-x_\pm$, where $x_\pm$ are the zeroes of $p_\kappa(x)$. Therefore, $q_\kappa(x)$ behaves as follows:
\begin{itemize}
\item For $\kappa=0$, $q_\kappa(x)$ reduces to the straight line $y(x)=4+8\h(0)x$, which has the point $x=-1/(2\h(0))$ as zero. In particular, $q_\kappa(x)>0,\ \forall x>0$.
\item For $\kappa<0$, $z_-,z_+<0$, and $q_\kappa(x)>0$ for every $x>0$.
\end{itemize}

In conclusion, $\gamma'(0)$ points upwards at every $(x_0,0)\in\tm1$.

We have the following consequence in $\t1$: let $(x_0,0)$ be a point at the left-hand side of $e_0$, and choose $r>0$ such that $D((x_0,0),r)\cap\Gamma_1=\varnothing$. Let $\gamma(s)$ be an orbit in $D((x_0,0),r)$ such that $\gamma(0)=(x_0,0)$. Then, the sign of $y'(s)$ is constant and equal to the sign of $y'(0)$, which is positive. By connectedness, the sign of $y'(s)$ of any orbit in $\Lambda_3^+\cup\Lambda_4^+$ is also positive. The same holds for an orbit contained in $\Lambda_1^+\cup\Lambda_2^+$, but this time $y'(s)$ is negative.  A similar situation holds in $\tm1$ in the monotonicity regions $\Lambda_3^-$ and $\Lambda_4^-$ if $\tau>0$; we have $y'(s)>0$ for any orbit contained in those regions. 

Finally, we prove that $y'(s)<0$ for an orbit contained in $\Lambda_1^-$ or $\Lambda_2^-$ in $\tm1$ if $\tau>0$. By symmetry of the phase plane w.r.t. the axis $y=0$ it only suffices to prove it in the region $\Lambda_1^-$. Fix a point $(x_\infty,y_\infty)\in\Omega^+$ and let $\gamma(s_n)=(x(s_n),y(s_n))=(x_n,y_n)$ be a sequence of points in an orbit $\gamma(s)$ lying in $\Lambda_1^-$ and converging to $(x_\infty,y_\infty)$. Note that $(x_\infty,y_\infty)$ are both positive. From Equation \eqref{1ordersys} we see that the values $\gamma'(s_n)$ satisfy:
$$
\left(\begin{array}{c}
x'(s_n)\\
y'(s_n)
	\end{array}\right)=\left(\begin{array}{c}
y_n\\
\displaystyle{\frac{4-\kappa x_n^2-y_n^2(4-x_n^2(\kappa-8\tau^2))+8 x_n\h(y_n)\sqrt{1-(1+\tau^2 x_n^2)y_n^2}}{x_n \left(4+\kappa x_n^2\right) \left(1+\tau^2 x_n^2\right)}}
\end{array}\right).
$$
Since $(x_\infty,y_\infty)\in\Omega^+$ we get $y_\infty^2(1+\tau^2x_\infty^2)=1$. Taking limits we see that the value $y'(s_n)$ converge to $-(1-y_\infty^2)(4+\kappa x_\infty^2)$, which is negative. Because $y'(s)=0$ only happens at $\Gamma_{-1}$, by continuity we get $y'(s)<0$. Again, a connectedness argument ensures us that this condition is fulfilled for every orbit lying in $\Lambda_1^-$, and by symmetry also in $\Lambda_2^-$.


This completes the proof of Proposition \ref{direccionesmonotonia}.
\end{proof}


\subsection{The structure of the phase plane.}\label{structurephaseplane} In the previous section we focused in how the orbits behave in each monotonicity region, hence how they move through $\te$. In this section we exhibit further properties of the phase plane that determine the global and local behavior of an orbit $\gamma(s)$ as it approaches to some point, or tends to \emph{escape} from $\te$. These properties are strongly influenced by the underlying geometric problem.

We point out that since $e_0$ is a solution of Equation \eqref{1ordersys}, because $\h\in C^1$ and by uniqueness of the Cauchy problem, an orbit could converge to $e_0$ with the parameter $s\rightarrow\infty$. In fact, in \cite{BGM1} we constructed explicit examples converging \emph{directly} to $e_0$, that is without spiraling around it. However, this situation cannot happen if $\h$ is even, as detailed next.
\begin{pro}\label{orbitanoequilibrio}
An orbit $\gamma$ in $\Theta_1$ cannot converge to $e_0$ or $e_{-1}$ (for $\kappa>0$).
\end{pro}

\begin{proof}
Let us analyze the structure of the orbits around $e_0$. The fact that $\h$ is even implies that $\h'(0)=0$, and the linearized system of Equation \eqref{1ordersys} at $e_0$ is
$$
\left(\begin{array}{c}
u\\
v
\end{array}\right)'=
\left(\begin{matrix}
0&1\\
F(\kappa,\tau,\h(0))&0
\end{matrix}\right)
\left(\begin{array}{c}
u\\
v
\end{array}\right),
$$
where $F(\kappa,\tau,\h(0))$ is a negative expression only depending on the mentioned variables; the fact that $\h\in\c$ is key here and in how the element $a_{22}$ in the linearized matrix vanishes. 

Hence, the orbits of the linearized system around the origin are ellipses, and by classical theory of non-linear autonomous systems we have two possible configurations around $e_0$: either the curves are closed, or they spiral around $e_0$, converging to it. However, the latter possibility cannot occur since the phase plane $\Theta_1$ is symmetric with respect to $y=0$, and $e_0$ belongs to this axis. In particular, all the orbits in $\Theta_1$ stays at a positive distance from $e_0$.

This proof carries over verbatim for the equilibrium point $e_{-1}$ of the phase plane $\tm1$ when $\kappa>0$.
\end{proof}

Next, we analyze the boundary points of $\te$ that cannot be limit points of orbits.

\begin{pro}\label{noorbitaborde}
An orbit $\gamma$ cannot converge to a point $(0,y)\in\overline{\Theta_\varepsilon},\ |y|<1$.
\end{pro}

\begin{proof}
By contradiction, suppose that such an orbit exists and let $\alpha(s)=(x(s),0,z(s))$ be its associated profile curve. Then, $\gamma(s_n)\rightarrow(0,y),\ y\in(-1,1)$, for a sequence $s_n$. This implies that $x(s_n)\rightarrow0$ and $x'(s_n)\rightarrow y\in (-1,1)$, that is $\alpha(s)$ approaches to the $E_3$-axis in a non-orthogonal way. 

Let $\sig$ be the $\Hs$ obtained by rotating $\alpha$. By the monotonicity properties of the phase plane, we see that a piece of $\sig$ can be written as a graph $z=u(x,y)$ over a punctured disk $D(\textbf{o},\delta)$ in $\M^2(\kappa)$. Moreover, the mean curvature $H(x,y)$ viewed as a function in $D(\textbf{o},\delta)-\{\textbf{o}\}$ extends continuously to $\textbf{o}$ with value $\h(y)$. By the work of Leandro and Rosenberg \cite{LeRo}, we know that $\sig$ extends smoothly to $D$ having vertical unit normal at $u(0,0)$ and hence angle function equal to $\pm1$. This contradicts the fact that $|y|<1$.
\end{proof}

In particular, if an $\Hs$ intersects the axis of rotation it does in an orthogonal way, i.e. $x'=\pm1$. Also, Proposition \ref{noorbitaborde} ensures us that an orbit $\gamma(s)$ can only converge to points in the boundary of $\Theta_\varepsilon$ located in $\Omega$. We show that $\gamma(s)$ cannot converge to some $(x_0,y_0)\in\Omega$ for the value of the parameter $s\rightarrow\pm\infty$. Indeed, if $(x(s),y(s))\rightarrow(x_0,y_0)\in\Omega$ for $s\rightarrow\pm\infty$, then the mean value theorem ensures us that $x'(s)\rightarrow0$. But $x'(s)=y(s)\rightarrow y_0\neq 0$, reaching to a contradiction. Hence, if $\gamma(s)$ converges to $\Omega$, reaches it at a finite instant.


Recall that the points $(0,\pm1)$ are missed by Proposition \ref{noorbitaborde}. In virtue of Proposition \ref{problemadirichlet} we can construct radial $\Hss$ intersecting orthogonally the axis of rotation over a small enough disk $D(0,\delta)$. These surfaces are defined next.
\begin{defi}\label{defisigmamasmenos}
Let be $\h\in\c$. We define the rotational $\h$-surface $\sig^+$ (resp. $\sig^-$) as the upwards oriented (resp. downwards oriented) radial solution of Proposition \ref{problemadirichlet}. Moreover, $\sig^+$ and $\sig^-$ agree after a rotation around a horizontal geodesic.
\end{defi}
The existence of $\sig^+$ and $\sig^-$ has the following implication on the phase plane $\Theta_\varepsilon$.

\begin{pro}\label{orbitasmasmenos}
There exists a unique orbit $\gamma_+$ (resp. $\gamma_-$) in $\t1$ having the point $(0,1)$ (resp. $(0,-1)$) in $\overline{\t1}$ as endpoint. Moreover, $\gamma_+$ and $\gamma_-$ are symmetric with respect to the axis $y=0$. There do not exist such orbits in $\tm1$.
\end{pro}

\begin{proof}
If prove the existence of $\gamma_+$, the existence of $\gamma_-$ follows from the symmetry of the phase plane $\Theta_1$.

In virtue of Lemma \eqref{problemadirichlet}, $\sig^+$ is described as the rotation of a curve $\alpha(s)=(x(s),0,z(s))$ around the $E_3$-axis and such that $x(0)=0$ and $x'(0)=1$. Since the mean curvature of $\sig^+$ at $p_0$ is positive and $\sig^+$ is upwards oriented, the mean curvature comparison principle yields $z(s)>0$ and $z'(s)>0$ for $s>0$ small enough. Thus, the orbit $\gamma_+(s)$ generated by $\alpha(s)$ belongs to $\Theta_1$ for $s>0$ small enough.

The mean curvature comparison principle ensures us that such an orbit cannot exist in $\Theta_{-1}$ for upwards oriented graphs. By symmetry of the phase plane we ensure that this condition also holds at the point $(0,-1)$, generating an orbit $\gamma_-$ that corresponds to the $\h$-surface $\sig^-$. Again, $\gamma_-$ cannot exist in $\Theta_{-1}$ by the mean curvature comparison principle.
\end{proof}

Thus, the orbit $\gamma_+(s)$ starts at the point $(0,1)\in\overline{\Theta_1}$, say at the instant $s=0$, and then is strictly contained in the monotonicity region $\Lambda_1^+$ for $s>0$ small enough. We ensure that $\gamma_+(s)$ cannot intersect again the boundary component $y=1/\sqrt{1+\tau^2x^2}$. In general, we prove that the orbits in $\Theta_\varepsilon$ cannot have endpoints in $\Omega$ arbitrarily.

\begin{pro}\label{comportamientobordeorbitas}
Let $\gamma(s)$ be an orbit in $\te$ and consider $\alpha(s)=(x(s),0,z(s))$ the associated arc-length parametrized curve. Suppose that $\gamma(s)$ has some $(x_0,y_0)\in\Omega^\pm$ as endpoint at $s=s_0$. Then, 

\begin{itemize}
\item[1.] If $\gamma(s_0)\in\Omega^+$, $z(s)$ has a local minimum at $s=s_0$. In this case, $\gamma(s)$ lies in $\Theta_1$ for $s>s_0$. For $s<s_0$, $\gamma(s)$ belongs to $\Theta_{-1}$.
\item[2.] If $\gamma(s_0)\in\Omega^-$, $z(s)$ has a local maximum at $s=s_0$. In this case, $\gamma(s)$ lies in $\Theta_1$ for $s<s_0$. For $s>s_0$, $\gamma(s)$ belongs to $\Theta_{-1}$.
\end{itemize}
\end{pro}

\begin{proof}
Suppose that $\gamma_+(s_0)=(x_0,y_0)\in\Omega^+$ for some $s_0$ and let $\alpha(s)=(x(s),0,z(s))$ be the arc-length parametrized curve defined by $\gamma(s)$. Note that $x'(s_0)=y_0>0$. Because $\gamma(s_0)\in\Omega^+$ we have that the angle function $x'(s)=y(s)$ satisfies $y(s_0)=1/\sqrt{1+\tau^2 x(s_0)^2}$, and thus the arc-length condition
$$
(1+\tau^2 x(s)^2)x'(s)^2+\frac{(4+\kappa x(s)^2)^2}{16}z'(s)^2=1
$$
ensures us that $z'(s_0)=0$. From Equation \eqref{eqmedia} and the fact that $z'(s_0)=0$, we get
$$
2\h(x'(s_0))=\frac{\left(4+\kappa x(s_0)^2\right)^2 16 x'(s_0)\left(1+\tau^2 x(s_0)^2\right)}{4\left(16 x'(s_0)^2 \left(1+\tau^2 x(s_0)^2\right)\right)^{3/2}}z''(s_0).
$$
Since $\h$ and $x'(s_0)=y(s_0)$ are positive, we see that $z''(s_0)$ is positive as well which yields that $z(s_0)$ is a local minimum of $z(s)$. Thus, $z(s)$ is increasing for $s>s_0$ and decreasing for $s<s_0$. This behavior implies that the orbit describing $\alpha(s)$ lies in $\Theta_1$ for $s>s_0$ and in $\Theta_{-1}$ for $s<s_0$. Hence, this orbit in $\Theta_{-1}$ \emph{ends} at $(x_0,y_0)\in\Omega^+$ and then \emph{starts} at the same point but this time in $\Theta_1$.

If $\gamma(s_0)\in\Omega^-$, the proof is similar; just note that $x'(s_0)=y_0<0$, hence $z''(s_0)$ is negative. This time, the orbit in $\Theta_1$ \emph{ends} at $\Omega^-$ and then \emph{starts} at the same point but this time in $\Theta_{-1}$.
\end{proof}
%

In any of these situations, i.e. where $z'(s)=0$ and it changes its monotony, we will say that the orbit in $\Theta_\varepsilon$ \emph{continues} in $\Theta_{-\varepsilon}$. and this continuation has to be understood as the extension of the associated $\Hs$ having a common point with the same unit normal.

Finally, we focus on whether an orbit can \emph{escape} from the phase plane $\te$.
\begin{pro}\label{orbitasacotadas}
Let be $\h\in\c$ and $\gamma(s)=(x(s),y(s))$ an orbit in $\te$. Then $x(s)$ cannot diverge to $\infty$ if $\kappa=0$, or tend to $2/\sqrt{-\kappa}$ if $\kappa<0$.
\end{pro}

\begin{proof}
The proof will be done by contradiction and distinguishing the possible values of $\kappa$.

\textbf{Case $\kappa<0$.} 

Let $\alpha(s)=(x(s),0,z(s))$ be the arc-length parametrized curve generated by $\gamma(s)$. Then, the fact that $\alpha(s)$ is the profile of a rotational $\Hs$ is equivalent to the following system to be fulfilled
\begin{equation}\label{sistemaode}
\left\lbrace\begin{array}{l}
\vspace{.25cm}x'(s)=\displaystyle{\frac{\cos\theta(s)}{\sqrt{1+\tau^2 x(s)^2}}},\\
z'(s)=\displaystyle{\frac{4\sin\theta(s)}{4+\kappa x(s)^2}},\\
\theta'(s)=\displaystyle{\frac{1}{(4+\kappa x(s)^2)\sqrt{1+\tau^2x(s)^2}}\left(8\h\left(\frac{\cos\theta(s)}{\sqrt{1+\tau^2x(s)^2}}\right)-\frac{4-\kappa x(s)^2}{x(s)}\sin\theta(s)\right)}.
\end{array}\right.
\end{equation}
Here, $\theta(s)$ is the angle that $\alpha'(s)$ makes with the $E_1$-direction.

Assume that $\gamma(s)$ lies in $\t1$ with $x(s)\rightarrow 2/\sqrt{-\kappa},\ s\nearrow s_0$; the case where $s\searrow s_0$ or $\gamma(s)$ lies in $\tm1$ are proved similarly. Hence, $\gamma(s)$ is contained in $\Lambda_1^+$ and stays there as $x(s)\rightarrow 2/\sqrt{-\kappa}$ and $y(s)\rightarrow y_0$, where $y_0\in[0,1/\sqrt{1-4\tau^2/\kappa}]$. In particular, $\theta(s)\rightarrow\theta_0\in[0,\pi/2]$, where $\theta_0$ is defined by $\cos\theta_0=y_0\sqrt{1-4\tau^2/\kappa}$. Because $\h\in\c$, at the limit $x_0:=2/\sqrt{-\kappa}$ we have
$$
\frac{2}{\sqrt{-\kappa}}\h\left(\frac{\cos\theta_0}{\sqrt{1+\tau^2x_0^2}}\right)>\sqrt{1-\left(\frac{\cos\theta_0}{\sqrt{1+\tau^2x_0^2}}\right)^2}=\sqrt{\frac{\sin^2\theta_0+\tau^2x_0^2}{1+\tau^2x_0^2}}.
$$
Now, it is trivial to check that the following inequality holds
$$
\h\left(\frac{\cos\theta_0}{\sqrt{1+\tau^2x_0^2}}\right)-\frac{\sqrt{-\kappa}}{2}\sin\theta_0\geq\h\left(\frac{\cos\theta_0}{\sqrt{1+\tau^2x_0^2}}\right)-\frac{\sqrt{-\kappa}}{2}\sqrt{\frac{\sin^2\theta_0+\tau^2x_0^2}{1+\tau^2x_0^2}}:=c_0>0.
$$
Taking limit we conclude
$$
\lim_{s\rightarrow s_0} 8\h\left(\frac{\cos\theta}{\sqrt{1+\tau^2x^2}}\right)-\frac{4-\kappa x^2}{x}\sin\theta=8\left(\h\left(\frac{\cos\theta_0}{\sqrt{1+\tau^2x_0^2}}\right)-\frac{\sqrt{-\kappa}}{2}\sin\theta_0\right)\geq 8c_0>0.
$$
So, the third equation in \eqref{sistemaode} yields that for $x$ close to $2/\sqrt{-\kappa}$ we have $\theta'(s)\geq c_1/(4+\kappa x(s)^2)$, for another positive constant $c_1$. In particular, $\lim_{s\rightarrow s_0}\theta'(s)=\infty$.

Now, we see the angle $\theta$ as a function $\theta(x)$ of $x$; this can be done since $x'(s)>0$ and by the inverse function theorem. Hence,
$$
\theta'(s)=\frac{d\theta}{ds}=\frac{d\theta}{dx}\frac{dx}{ds}=\theta'(x)x'(s)=\theta'(x)\frac{\cos\theta(x)}{\sqrt{1+\tau^2 x^2}}.
$$
In this setting, the third equation in \eqref{sistemaode} for $x$ close enough to $2/\sqrt{-\kappa}$ reads as:
$$
\theta'(x)\cos\theta(x)=\frac{1}{4+\kappa x^2}\left(8\h\left(\frac{\cos\theta(x)}{\sqrt{1+\tau^2 x^2}}\right)-\frac{4-\kappa x^2}{x}\sin\theta(x)\right)\geq\frac{c_1}{4+\kappa x^2}.
$$
Integrating from some fixed, large enough $x_0$ and $x$ yields
$$
\sin\theta(x)\geq\frac{c_1}{2\sqrt{-\kappa}}\arctanh\left(\frac{\sqrt{-\kappa}}{2}x\right)+\sin\theta(x_0).
$$
This is a contradiction since the right-hand side tends to infinity as $x$ approaches to $2/\sqrt{-\kappa}$.

Following the idea developed for the case $\kappa<0$, the case $\kappa=0$ is proved similarly. Now, the $x$-coordinate of an orbit $\gamma(s)=(x(s),y(s))\in\te$ would satisfy $x(s)\rightarrow\infty$. We give a sketch of the proof.

\textbf{Case $\kappa=0$.} Suppose that $\gamma(s)$ is in the phase plane $\t1$; the case $\tm1$ is analogous. Moreover, we assume that $\gamma(s)$ is strictly contained in $\Lambda_1^+$, since the case when $\gamma(s)$ lies in $\Lambda_2^+$ is equivalent by symmetry of $\t1$. After expressing $\theta$ as a function of $x$, the third equation in \eqref{sistemaode} for $x$ large enough yields 
$$
\theta'(x)\cos\theta(x)=2\h\left(\frac{\cos\theta(x)}{\sqrt{1+\tau^2x^2}}\right)-\frac{\sin\theta(x)}{x}\geq c_0>0.
$$
Explicit integration from some $x_0$ large enough and $x>x_0$ yields $\sin\theta(x)\geq c_0x+\sin\theta(x_0)$, which is obviously a contradiction when $x\rightarrow\infty$.
\end{proof}

\section{Proof of Theorem \ref{teoremadelaunay}}\label{clasificacionrotacionales}

We will use the coordinate model $\rkt$ as introduced in Section \ref{secpropbasicas}. We suppose that the axis of rotation is the $E_3$-axis $(0,0,z),\ z\in\R$.

The proof will be done by taking advantage from the phase plane analysis made in Section \ref{rotational}. The structure of the phase planes $\Theta_\varepsilon,\ \varepsilon=\pm1$, for the different cases $\kappa,\tau$ was shown in Figure \ref{planosfases} and detailed in Proposition \ref{direccionesmonotonia}. In particular, the equilibrium point $e_0$ in $\Theta_1$ given by Equation \eqref{ecuequilibrio} generates a CMC vertical cylinder, obtaining the first example of the classification.

For the existence of the $\h$-sphere, consider the orbit $\gamma_+$ starting at the point $(0,1)$ at the instant $s=0$ given by Proposition \ref{orbitasmasmenos}. For $s>0$ small enough $\gamma_+(s)$ lies in $\Lambda_1^+$, and from Propositions \ref{comportamientobordeorbitas} and \ref{orbitasacotadas} we conclude that $\gamma_+(s)$ is contained in $\Lambda_1^+$ until it intersects the axis $y=0$ at some $\gamma(s_0)=(r_0,0),\ r_0>0$. This also holds for the orbit $\gamma_-(s)$, i.e. this time $\gamma_-(s)$ ends at the point $(0,-1)$ at some instant $s_1>0$, is contained in $\Lambda_2^+$ and comes from intersecting the axis $y=0$ at a finite instant. Since $\gamma_+$ and $\gamma_-$ are symmetric, they meet orthogonally at the axis $y=0$. By uniqueness $\gamma_+$ and $\gamma_-$ can be smoothly glued to generate a compact orbit $\gamma_0:=\gamma_+\cup\gamma_-$ that joins the points $(0,1)$ and $(0,-1)$. 

\begin{figure}[H]
\centering
\includegraphics[width=.7\textwidth]{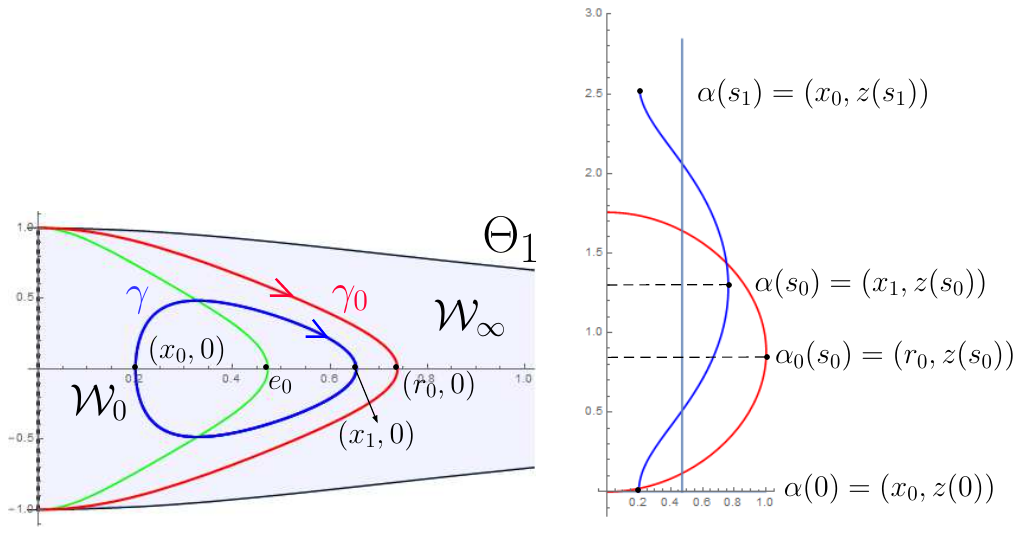}  
\caption{Left: the phase plane $\t1$ and the orbits corresponding to the $\h$-sphere, the CMC cylinder and an $\h$-unduloid. Right, the profile curves of these $\h$-surfaces.}
\label{fasesesferadelaunay}
\end{figure}

The arc-length parametrized curve $\alpha_0(s)$ associated to $\gamma_0(s)$ intersects the $E_3$-axis at the instant $s=0$, has strictly increasing height function (hence is embedded in $\rkt$), and its angle function at the instant $s=s_0$ vanishes; here, the $x(s)$-coordinate reaches a global maximum and then decreases. By the even condition on $\h$, $\alpha_0$ is symmetric with respect to the horizontal geodesic at height $z(s_0)$ in the $xz$-plane. The $\h$-surface generated by rotating $\alpha_0$ is an $\h$-sphere, denoted by $S_\h$, with strictly decreasing angle function. See Figure \ref{fasesesferadelaunay}.

Recall that the orbit $\gamma_0$ of the $\h$-sphere divides $\Theta_1$ in two connected components: one bounded containing the equilibrium $e_0$, that we will denote by $\mathcal{W}_0$, and other unbounded that we will denote by $\mathcal{W}_\infty$. 

To prove the existence of the $\h$-unduloids, consider $x_0>0$ such that $(x_0,0)\in\mathcal{W}_0$, and suppose that $(x_0,0)$ is at the left-hand side of $e_0$. Let $\gamma(s)=(x(s),y(s))$ be the orbit passing through $(x_0,0)$ at the instant $s=0$.  Then, for $s>0$ small enough $\gamma(s)$ is contained in $\Lambda_4^+$ and follows its monotonicity direction. By continuity, $\gamma(s)$ has to intersect $\Gamma_1$ at some finite point, where the $y(s)$-coordinate of $\gamma(s)$ reaches a maximum, and then $\gamma(s)$ enters to the region $\Lambda_1^+$. Since $\gamma(s)$ cannot intersect $\gamma_0$ by uniqueness of the Cauchy problem and $\gamma(s)$ cannot converge to $e_0$ with $s\rightarrow\infty$ in virtue of Proposition \ref{orbitanoequilibrio}, the only possibility for $\gamma(s)$ is to intersect the axis $y=0$ at some finite point $\gamma(s_0)=(x_1,0)$ lying at the right-hand side of $e_0$. By symmetry of $\Theta_1$, the same behavior holds in the regions $\Lambda_2^+$ and $\Lambda_3^+$ and then $\gamma(s)$ reaches again the point $(x_0,0)$ at some instant $s_1>0$, which implies that $\gamma(s)$ is a periodic orbit. Note that in any case, $\gamma(s)$ cannot converge to the segment $\{0\}\times[-1,1]$ in virtue of Proposition \ref{noorbitaborde}.


This orbit generates an arc-length parametrized curve $\alpha(s)$ whose height function is strictly increasing (since $\gamma(s)\subset\Theta_1)$ and hence $\alpha(s)$ is embedded. The $x(s)$-coordinate of $\alpha(s)$ is periodic, its maximum is $x(s_0)=x_1$ and its minimum is $x(0)=x_0$; see Figure \ref{fasesesferadelaunay}. The rotation of $\alpha(s)$ generates a properly embedded $\h$-surface $U_\h$ that is diffeomorphic to $\S^1\times\R$ and with periodic distance to the $E_3$-axis, that is, $U_\h$ is an $\h$-unduloid. 

Moreover, each $\h$-unduloid is uniquely determined by the value $x_0$ that agrees with the radius of the smallest circumference contained in $U_\h$. Thus, the family of $\h$-unduloids is a continuous family $\{U_\h(r)\}$ parametrized by the \emph{necksize} of their \emph{waists}, where $0<r<2(\sqrt{4\h(0)^2+\kappa}+2\h(0))^{-1}$. Similarly to the CMC case, when $r\rightarrow 2(\sqrt{4\h(0)^2+\kappa}+2\h(0))^{-1}$ the $\h$-unduloids $U_\h(r)$ converge to the vertical cylinder of CMC $\h(0)$, and when $r\rightarrow 0$ the $\h$-unduloids converge to a singular chain of tangent $\h$-spheres.

Lastly, we prove the existence of the $\h$-nodoids. For that, let $(r_0,0)$ be the point of intersection of $\gamma_0$ with $y=0$, and fix some $x_0>r_0$. Consider the orbit $\gamma(s)$ passing through $(x_0,0)$ at the instant $s=0$. For $s>0$ small enough, $\gamma(s)$ lies in $\Lambda_2^+$, and since $\gamma_0$ and $\gamma$ cannot intersect each other, $\gamma(s)$ has some $\gamma(s_0)=(x_1,y_1)\in\Omega^-,\ x_1>0,\ y_1<0$ as endpoint. By symmetry, $\gamma(s)$ for $s<0$ has the same behavior at $\Lambda_1^+$, having the point $\gamma(-s_0)=(x_1,-y_1)\in\Omega^+$ as endpoint. This orbit generates an arc-length parametrized curve $\alpha(s)$ which is also symmetric with respect to the rotation of angle $\pi$ around the horizontal geodesic in the $xz$-plane at height $z(0)$; after a vertical translation we can suppose that $z(0)=0$. At this height, the $x(s)$-coordinate of $\alpha(s)$ reaches its maximum $x_0$, and then decreases to the value $x_1$. The height $z(s)$ reaches a minimum at $s=-s_0$ where $z'(-s_0)=0$, then increases and reaches a maximum at $s=s_0$ where again $z'(s_0)=0$. See Figure \ref{nodoidefig}.


\begin{figure}[H]
\centering
\includegraphics[width=.7\textwidth]{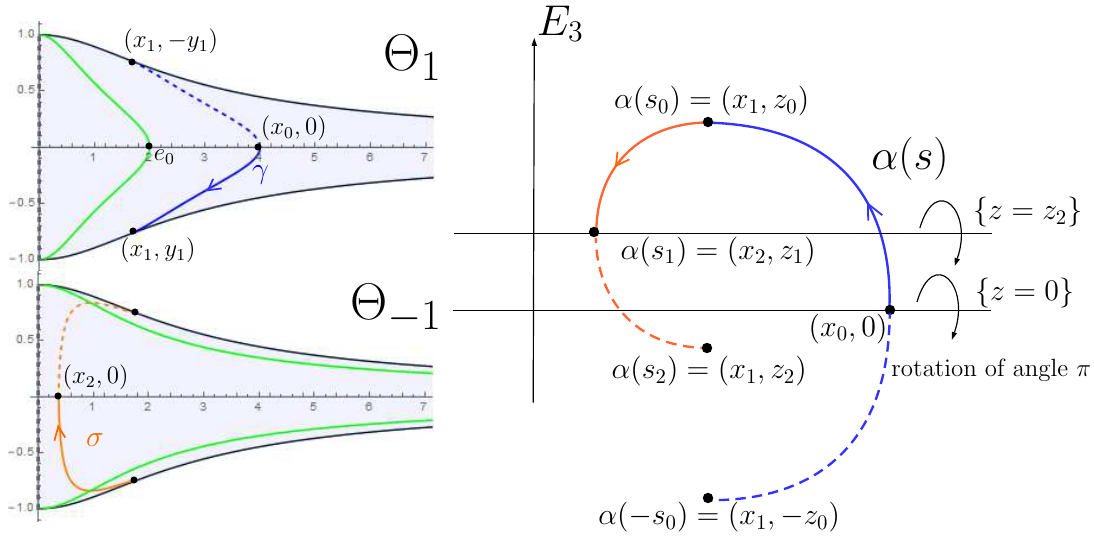}  
\caption{Left: the phase plane and the orbit generating an $\h$-nodoid. Right: the profile curve of the $\h$-nodoid.}
\label{nodoidefig}
\end{figure}

Now, for $s>s_0$ the function $z(s)$ is decreasing, hence $\alpha(s)$ for $s>s_0$ generates an orbit $\sigma(s)$ in $\Theta_{-1}$ having the point $(x_1,y_1)\in\Omega^-$ as endpoint. The monotonicity properties of $\Theta_{-1}$ and Proposition \ref{noorbitaborde} ensures us that $\sigma(s)$ has to intersect the axis $y=0$ at some finite point $\sigma(s_1)=(x_2,0),\ x_2>0$. By symmetry, $\sigma(s)$ ends up having the point $(x_1,-y_1)\in\Omega^+$ as endpoint for some $s=s_2$. Thus, the height of $\alpha(s)$ is strictly decreasing starting at the value $z_0$ and having the value $z(s_2)$ as minimum. This time, the distance $x(s)$ to the $E_3$-axis has the value $x_2$ as minimum. See again Figure \ref{nodoidefig}.
 
\section{Proof of Theorem \ref{teorematoros}}\label{discusiontoros}

\subsection{Non-existence of rotational $\h$-tori.}\label{nonexistencetori} We show the non-existence of rotational $\h$-tori for further prescribed functions $\h\in\c$ that generalize the case that $\h$ is constant.


Suppose that $\alpha(s)=(x(s),0,z(s))$ in $\rkt$ is arc-length parametrized with respect to the metric \eqref{metricappa} and that generates an $\Hs$ after being rotated around the $E_3$-axis. Then, as pointed out in Proposition \ref{orbitasacotadas}, its coordinates satisfy Equation \eqref{sistemaode}.

Now, suppose that $\alpha(s)$ generates a compact portion of an $\h$-nodoid as in showed in Figure \ref{nodoidefig}. Hence, $\alpha(s)$ is defined for $s\in(0,s_1)$, and satisfies $\theta(0)=\pi/2,\ \theta(s_0)=\pi$ and $\theta(s_1)=3\pi/2$. In particular, $\theta'(s)>0$ for every $s\in(0,s_1)$. After a vertical translation, suppose that $z(0)=0$.

From system \eqref{sistemaode}, we have that the function $z(s)$ can be written by the integral formula $z(s)=\int_{0}^s4\sin\theta(t)/(4+\kappa x(t)^2)dt$. Because $\theta'(s)>0$, we can express the functions $x(s),z(s)$ and $\nu(s)$ as functions of the angle $\theta$. The chain rule yields
$$
\frac{dz}{d\theta}=\frac{z'(s)}{\theta'(s)}=\frac{4\sin\theta\sqrt{1+\tau^2x^2}}{8\h\left(\nu\right)-\frac{4-\kappa x^2}{x}\sin\theta}.
$$
Recall that the denominator has the same sign of $\theta'(s)$, hence it is positive everywhere. Thus, we have
\begin{equation}\label{integrales}
I_1:=-\int_\pi^{3\pi/2}\frac{dz}{d\theta}d\theta=z(s_0)-z(s_1);\hspace{.5cm} I_2:=\int_{\pi/2}^\pi\frac{dz}{d\theta}d\theta=z(s_0)-z(0).
\end{equation}
and so $z(s_1)>z(0)$ if and only if $I_1<I_2$. 

Choose $s\in (0,s_0)$ and $\overline{s}\in(s_0,s_1)$ such that $\sin\theta(s)=-\sin\theta(\overline{s})$; in particular, $x(s)>x(\overline{s})$. Since $\kappa\leq0$, the map $x\mapsto f(x):=(4-\kappa x^2)/x$ is positive and strictly decreasing for $x<2/\sqrt{-\kappa}$; when $\kappa=0$ this is fulfilled for every $x>0$. As a matter of fact, $f(x(s))<f(x(\overline{s}))$. Finally, since $\nu(s)=\cos\theta(s)/\sqrt{1+\tau^2x(s)^2}$, after a straightforward computation we conclude that $\nu(s)>\nu(\overline{s})$.

Now, suppose that $\h\in\c$ is a non-increasing function in the interval $[-1,0]$. In particular, for $s,\overline{s}$ as above we have $\h(\nu(s))\leq\h(\nu(\overline{s}))$. Bearing these discussions in mind, the following inequality holds
\begin{equation}\label{desigualdadtoros}
-\frac{dz}{d\theta}(\overline{s})=\frac{-4\sin\theta(\overline{s})\sqrt{1+\tau^2 x(\overline{s})^2}}{8\h(\nu(\overline{s}))-f(x(\overline{s}))\sin\theta(\overline{s})}<\frac{4\sin\theta(s)\sqrt{1+\tau^2 x(s)^2}}{8\h(\nu(s))-f(x(s))\sin\theta(s)}=\frac{dz}{d\theta}(s),
\end{equation}
and hence
$$
z(s_0)-z(s_1)=-\int_\pi^{3\pi/2}\frac{dz}{d\theta}d\theta<\int_{\pi/2}^\pi\frac{dz}{d\theta}d\theta=z(s_0)-z(0);
$$
i.e. $z(s_1)>z(0)$ and no $\h$-tori exist.

Note that the particular choice $\h=H_0\in\R$ is compiled in this case, recovering the non-existence of rotational CMC tori in $\mathbb{H}^2(\kappa)\times\R$, $\nil$ and $\sl$.

\subsection{Existence of rotational $\h$-tori.}\label{existencetori} Finally, we exhibit the existence of rotational $\h$-tori. The fact that $\tau>0$ has strong implications on the behavior of the angle function of a rotational $\Hs$. This, along with the arbitrariness of the prescribed function $\h$, was the key that suggested us that the existence of rotational $\h$-tori could be possible in these spaces.

Following the same notation as in Section \ref{nonexistencetori}, consider the compact piece of the $\h$-nodoid generated by $\alpha(s)=(x(s),0,z(s))$ for $s\in(0,s_1)$ and such that $\theta(0)=\pi/2,\ \theta(s_0)=\pi$ where $x(s_0):=x_1$ and $z'(s_0)=0$, and $\theta(s_1)=3\pi/2$.


Again, since $\theta'(s)>0$ we express $\nu(s)$, $z(s)$ and $x(s)$ as functions of the angle $\theta$. The motion of the orbits in the phase plane $\tm1$ implies that $\nu(\theta)$ behaves as follows: $\nu(\theta)$ is strictly decreasing for $\theta\in(0,\pi/2)$, and when reaches the value $\theta=\pi$ it keeps decreasing until reaching a global minimum at some $\hat{\theta}\in(\pi,3\pi/2)$. Then, $\nu(\theta)$ is strictly increasing until reaching the value $\theta=3\pi/2$, where it vanishes again.

Given $\h\in\c$, we name $I_1(\h),I_2(\h)$ to the integrals of Equation \eqref{integrales} for the prescribed function $\h$. Let $H_0$ be a positive constant. Since Equation \eqref{desigualdadtoros} holds for $\h=H_0$, we conclude that $I_2(H_0)>I_1(H_0)$. Let us fix some $\delta>0$ small enough (we will determine later a bound for $\delta$), $x_1>0$ and define $\nu_0:=-1/\sqrt{1+\tau^2 x_1^2}$. For each $\lambda>H_0$, consider the following function:
$$
\h_\lambda(y)=\left\lbrace\begin{array}{lll}
H_0 & \mathrm{if} & y\in[-1,\nu_0-\delta]\\
\lambda & \mathrm{if} & y\in[\nu_0,-\nu_0]\\
H_0 & \mathrm{if} & y\in[-\nu_0+\delta,1],
\end{array}\right.
$$
and extend it to be $C^1$ in the whole interval $[-1,1]$.

If we prove that $I_2(\h_{\lambda_*})<I_1(\h_{\lambda_*})$ for some $\lambda_*>H_0$, by continuity we would have $I_2(\h_{\lambda_0})=I_1(\h_{\lambda_0})$ for some $\lambda_0\in(H_0,\lambda_*)$, and hence we conclude the existence of an $\h_{\lambda_0}$-torus. We prove this next.

We will denote with a sub-index $(\cdot)_\lambda$ to the functions corresponding to the prescribed function $\h_\lambda$, for each $\lambda$. Since $\sin\theta>0$ for $\theta\in(\pi/2,\pi)$, $x_\lambda(\theta)>x_0$, and $\h_\lambda(\nu_\lambda(\theta))=\lambda$, the following holds
$$
I_2(\h_\lambda)=\int_{\pi/2}^\pi\frac{4\sin\theta\sqrt{1+\tau^2x_\lambda^2}}{8\h_\lambda(\nu_\lambda)-f(x_\lambda)\sin\theta}d\theta>\int_{\pi/2}^\pi\frac{4\sin\theta\sqrt{1+\tau^2x_0^2}}{8\h_\lambda(\nu_\lambda)}d\theta=\sqrt{1+\tau^2x_0^2}\int_{\pi/2}^\pi\frac{4\sin\theta}{8\lambda}d\theta.
$$
Hence, $I_2(\h_\lambda)\rightarrow 0$ as $\lambda\rightarrow\infty$.

On the other hand, let be $\theta_1<\theta_2\in(\pi,3\pi/2)$ such that $\nu_{H_0}(\theta_i)=\nu_0-\delta$; we choose $\delta$ small enough so that the minimum of $\nu_{H_0}$ is smaller than $\nu_0-\delta$, hence $\theta_i$ are well defined. So, we can split $I_1(\h_\lambda)$ as
$$
-\int_\pi^{\theta_1}\frac{4\sin\theta\sqrt{1+\tau^2x_\lambda^2}}{8\h_\lambda(\nu_\lambda)-f(x_\lambda)\sin\theta}d\theta-\int_{\theta_1}^{\theta_2}\frac{4\sin\theta\sqrt{1+\tau^2x_\lambda^2}}{8\h_\lambda(\nu_\lambda)-f(x_\lambda)\sin\theta}d\theta-\int_{\theta_2}^{3\pi/2}\frac{4\sin\theta\sqrt{1+\tau^2x_\lambda^2}}{8\h_\lambda(\nu_\lambda)-f(x_\lambda)\sin\theta}d\theta.
$$
Note that for $\theta\in(\theta_1,\theta_2)$ in the second integral, the function $\h_\lambda(\nu_\lambda(\theta))$ is equal to $H_0$. In particular, all the geometric quantities appearing are the ones corresponding for the constant function $H_0$. Hence,
$$
I_1(\h_\lambda)>\int_{\theta_1}^{\theta_2}\frac{4\sin\theta\sqrt{1+\tau^2x_\lambda^2}}{8\h_\lambda(\nu_\lambda)-f(x_\lambda)\sin\theta}d\theta=\int_{\theta_1}^{\theta_2}\frac{4\sin\theta\sqrt{1+\tau^2x_{H_0}^2}}{8H_0-f(x_{H_0})\sin\theta}d\theta=c_0>0,
$$
where $c_0$ is a positive constant that does not depend on $\lambda$.

Thus, for $\lambda$ big enough we have $I_1(\h_\lambda)>I_2(\h_\lambda)$, hence there exists some $\lambda_0>H_0$ such that $I_1(\h_{\lambda_0})=I_2(\h_{\lambda_0})$, i.e. there exist a rotational $\h_{\lambda_0}$-torus. See Figure  \ref{homotopiatoro}.

\begin{figure}[H]
\centering
\includegraphics[width=.7\textwidth]{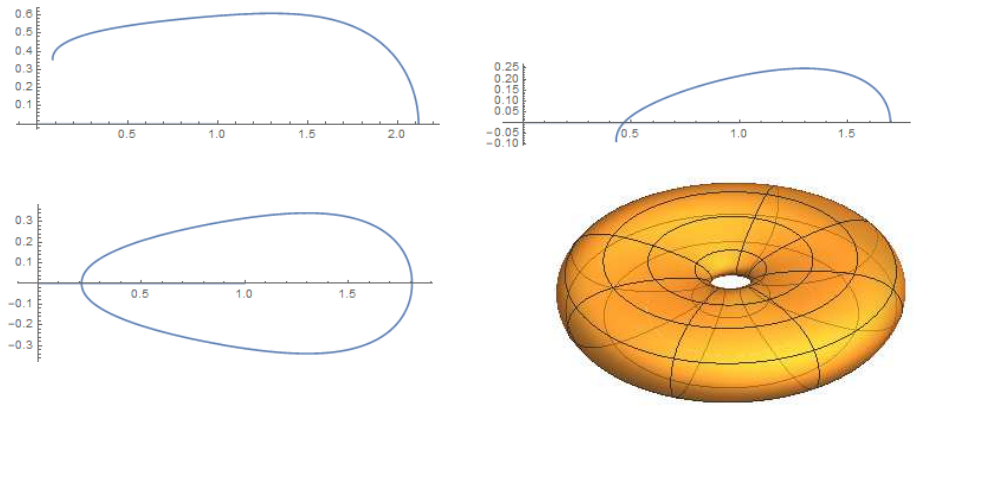}  
\caption{Top: the profile curves for different functions $\h_\lambda$, for which $I_1(\h_\lambda)-I_2(\h_\lambda)$ changes sign. Bottom: the profile curve and the rotational $\h_{\lambda_0}$-torus for a function $\h_{\lambda_0}$ such that $I_1(\h_{\lambda_0})=I_2(\h_{\lambda_0})$. Here, the space is $\nil$.}
\label{homotopiatoro}
\end{figure}

\def\refname{References}

The author was partially supported by MICINN-FEDER Grant No. MTM2016-80313-P.


\begin{thebibliography}{9}
\vspace{-.5cm}

\bibitem[AbRo1]{AbRo1} U. Abresch, H. Rosenberg, A Hopf differential for constant mean curvature surfaces in $\S^2\times\R$ and $\mathbb{H}^2\times\R$, {\it  Acta Math.} {\bf 193} (2004), 141--174.

\bibitem[AbRo2]{AbRo2} U. Abresch, H. Rosenberg, Generalized Hopf differentials, \emph{Mat. Contemp.} \textbf{28} (2005), 1--28.


\bibitem[Ale]{Ale} A.D. Alexandrov, Uniqueness theorems for surfaces in the large, I,  \emph{Vestnik Leningrad Univ.} {\bf 11} (1956), 5--17. (English translation): Amer. Math. Soc. Transl. {\bf 21} (1962), 341--354.


\bibitem[Bue1]{Bue1} A. Bueno, Translating solitons of the mean curvature flow in the space $\mathbb{H}^2\times \R$, {\it J. Geom.} \textbf{109} (2018)

\bibitem[Bue2]{Bue2} A. Bueno, Uniqueness of the translating bowl in $\mathbb{H}^2\times \R$, {\it J. Geom.} \textbf{11} (2020)

\bibitem[Bue3]{Bue3} A. Bueno, A Delaunay-type classification result for prescribed mean curvature surfaces in $\mathbb{M}^2(\kappa)\times \R$, {\it Pac. J. Math.}, \emph{to appear}, (2021).

\bibitem[Bue4]{Bue4} A. Bueno, Properly embedded surfaces with prescribed mean curvature in $\mathbb{H}^2\times\R$, \emph{Ann. Global Ann. Geom.} \textbf{59}, 69--80 (2021).


\bibitem[BGM1]{BGM1} A. Bueno, J.A. Gálvez, P. Mira, Rotational hypersurfaces of prescribed mean curvature, \emph{J. Differential Equations} \textbf{268} (2020), 2394--2413.

\bibitem[BGM2]{BGM2} A. Bueno, J.A. Gálvez, P. Mira, The global geometry of surfaces with prescribed mean curvature in $\R^3$, \emph{Trans. Amer. Math. Soc.} \textbf{373} (2020), 4437-4467.

\bibitem[Chr]{Chr} E.B. Christoffel, Über die Bestimmung der Gestalt einer krummen Oberfläche durch lokale Messungen auf derselben, {\it J. Reine Angew. Math.} {\bf 64} (1865), 193--209.

\bibitem[Dan]{Dan} B. Daniel, Isometric immersions into 3-dimensional homogeneous manifolds, \emph{Comment. Math. Helv.} \textbf{82} (2007), 87--131.

\bibitem[DHM]{DHM} B. Daniel, L. Hauswirth, P. Mira, Constant mean curvature surfaces in homogeneous manifolds, Korea Institute for Advanced Study, Seoul, Korea, 2009.


\bibitem[FeMi]{FeMi} I. Fern\'andez, P. Mira, Constant mean curvature surfaces in 3-dimensional Thurston geometries. In \emph{Proceedings of the International Congress of Mathematicians}, Volume II (Invited Conferences), pages 830--861. Hindustan Book Agency, New Delhi, 2010.

\bibitem[GaMi1]{GaMi1} J.A. Gálvez, P. Mira, A Hopf theorem for non-constant mean curvature and a conjecture of A.D. Alexandrov, {\it Math. Ann.} {\bf 366} (2016), 909--928.

\bibitem[GaMi2]{GaMi2} J.A. Gálvez, P. Mira, Uniqueness of immersed spheres in three-manifolds, \emph{J. Differential Geometry} \textbf{116} (2020), 459--480.

\bibitem[GaMi3]{GaMi3} J.A. Gálvez, P. Mira, Rotational symmetry of Weingarten spheres in homogeneous three-manifolds, \emph{J. Reine Angew. Math.} \textbf{773} (2021), 21--66.

\bibitem[Gor]{Gor} C. Gorodski, Delaunay-type surfaces in the $2\times2$ real unimodular group, \emph{Annali di Matematica.} \textbf{180} (2001), 211--221.

\bibitem[GuGu]{GuGu} B. Guan, P. Guan, Convex hypersurfaces of prescribed curvatures, {\it Ann. Math.} {\bf 156} (2002), 655--673.


\bibitem[HsHs]{HsHs} W. T. Hsiang, W. Y. Hsiang, On the uniqueness of isoperimetric solutions and imbedded soap bubbles in non-compact symmetric spaces, \emph{Invent. Math.} \textbf{85} (1989), 39--58.


\bibitem[LeRo]{LeRo} C. Leandro, H. Rosenberg. Removable singularities for sections of Riemannian submersions of prescribed mean curvature, {\it Bull. Sci. Math.} {\bf 133} (2009), 445--452.

\bibitem[LiMa]{LiMa} F. Martin, J. H. S. de Lira, Translating solitons in Riemannian products, \emph{J. Differential Equations} \textbf{266} (2019), 7780--7812.


\bibitem[Min]{Min} H. Minkowski, Volumen und Oberfläche, \emph{Math. Ann.} \textbf{57} (1903), 447--495.


\bibitem[PeRi]{PeRi} R. H. L. Pedrosa, M. Ritoré, Isoperimetric domains in the Riemannian product of a circle with a simply connected space form and applications to free boundary problems, \emph{Indiana Univ. Math. J.} \textbf{48} (1999), 1357--1394.

\bibitem[Pip]{Pip} G. Pipoli, Invariant translators of the Heisenberg group, preprint, arXiv:1811.04619.

\bibitem[Pog]{Pog} A.V. Pogorelov, Extension of a general uniqueness theorem of A.D. Aleksandrov to the case of nonanalytic surfaces (in Russian), {\it Doklady Akad. Nauk SSSR} {\bf 62} (1948), 297--299.

\bibitem[Tom]{Tom} P. Tomter,  Constant mean curvature surfaces in the Heisenberg group. \emph{Proc. Sympos. Pure Math.} \textbf{54} (1993), 485--495.

\bibitem[Tor]{Tor} F. Torralbo, Rotationally invariant constant mean curvature surfaces in homogeneous 3-manifolds. \emph{Diff. Geo. Appl.} \textbf{28} (2010), 593--607. 
\end{thebibliography}
\end{document}